\documentclass{amsart}
\usepackage{amssymb,latexsym,amsthm}
\usepackage[centertags]{amsmath}
\usepackage{amsfonts}
\usepackage{graphicx}

\DeclareMathOperator{\tr}{\rm{tr}}
\DeclareMathOperator{\ib}{\rm{b}_I}

\DeclareMathOperator{\SL}{\rm{SL}}
\DeclareMathOperator{\inter}{\rm{int}}
\DeclareMathOperator{\bd}{\partial}
\DeclareMathOperator{\cl}{cl}

\newcommand{\mc}{\mathcal}

\newcommand{\ol}{\overline}
\newcommand{\li}{\hat} 
\newcommand{\til}{\tilde}

\newcommand{\R}{\mathbb{R}}\newcommand{\N}{\mathbb{N}}
\newcommand{\Z}{\mathbb{Z}}\newcommand{\Q}{\mathbb{Q}}
\newcommand{\T}{\mathbb{T}}

\newcommand{\sm}{\setminus}

\newcommand{\ie}{i.e.\ }
\newcommand{\eg}{e.g.\ }

\newtheorem{theorem}{Theorem}[subsection]
\newtheorem{corollary}[theorem]{Corollary}
\newtheorem{lemma}[theorem]{Lemma}
\newtheorem{proposition}[theorem]{Proposition}

\newtheorem{question}[theorem]{Question}

\newtheorem*{theorem*}{Theorem}

\theoremstyle{definition}
\newtheorem{definition}[theorem]{Definition}

\theoremstyle{remark}
\newtheorem{remark}[theorem]{Remark}

\title{Aperiodic invariant continua for surface homeomorphisms}

\begin{document}

\begin{abstract}
We prove that if a homeomorphism of a closed orientable surface $S$ has no wandering points and leaves invariant a compact, connected set $K$ which contains no periodic points, then either $K=S=\T^2$, or $K$ is the intersection of a decreasing sequence of annuli. A version for non-orientable surfaces is  given. 
\end{abstract}

\author[A.  Koropecki]{Andres Koropecki}

\address{Universidade Federal Fluminense, Instituto de Matem\'atica, Rua M\'ario Santos Braga S/N, 24020-140 Niteroi, RJ, Brasil}

\email{koro@mat.uff.br}

\maketitle

\subsection{Introduction}

By \emph{aperiodic invariant continuum} we mean a compact connected set which is invariant by some homeomorphism of a compact surface, and which contains no periodic points.  We are interested in describing aperiodic invariant continua of non-wandering homeomorphisms. This type of sets appear frequently when studying generic area-preserving diffeomorphisms, due to a result of Mather \cite{mather-area}, which states that for such diffeomorphisms, the boundary of certain open invariant sets (see Definition \ref{def:regular}) is a finite union of aperiodic continua. Thus, having good topological information about aperiodic invariant continua is helpful to describe the dynamics of a $C^r$-generic area-preserving diffeomorphism. An example of this is the work of Franks and Le Calvez in \cite{franks-lecalvez} in the case that the surface is a sphere. 

Our main result is the following
\begin{theorem}\label{th:continuo} Let $f\colon S\to S$ be a homeomorphism of a compact orientable surface such that $\Omega(f)=S$. If $K$ is an $f$-invariant continuum, then one of the following holds:
\begin{enumerate}
\item $f$ has a periodic point in $K$;
\item $K$ is annular;
\item $K=S=\T^2$;
\end{enumerate}
\end{theorem}

By \emph{annular continuum} we mean an intersection of a nested sequence of topological annuli (see Definition \ref{def:annular}). When $S$ is non-orientable, a version of Theorem \ref{th:continuo} holds, however with two extra cases: $K$ could be a non-separating continuum in a M\"obius strip, and in the case that $K=S$, the surface could be $\T^2$ or the Klein bottle (Corollary \ref{coro:non-orientable}).

An important result of \cite{franks-lecalvez} states that for a generic area-preserving diffeomorphism of the sphere, the stable and unstable manifolds of hyperbolic periodic points are dense. This fact was generalized to an arbitrary surface by Xia \cite{xia-area}, and one of the main steps of his proof is obtaining a version of Theorem \ref{th:continuo} which assumes generic conditions on the (area-preserving) diffeomorphism and is restricted to continua which are the closure of a particular kind of open sets. Thus Theorem \ref{th:continuo} extends Xia's result to general homeomorphisms without wandering points (which includes area-preserving homeomorphisms), with no additional hypothesis on the continuum and no genericity conditions.

A question that motivates studying aperiodic invariant continua is the following:

\begin{question} \label{q:1} What are the possible obstructions to the transitivity of a $C^r$-generic area-preserving diffeomorphism?
\end{question}

Bonatti and Crovisier proved in \cite{bonatti-crovisier} that a $C^1$-generic area-preserving diffeomorphism of a compact manifold (of any dimension) is transitive. However, in dimension $2$, it is known that this is not true in the $C^r$ topology if $r$ is large enough, because of the KAM phenomenon: There are open sets of diffeomorphisms where a $C^r$-generic element has an elliptic periodic point surrounded by invariant circles (see, for instance, \cite{douady}), and this is an obstruction to transitivity. Hence the question is: is this the only possible obstruction? In other words, does the non-transitivity of a $C^r$-generic area-preserving diffeomorphism imply the existence of elliptic periodic points?

Studying Question \ref{q:1}, aperiodic invariant continua appear naturally as boundaries of invariant open sets. Theorem \ref{th:continuo} implies that the presence of annular periodic continua is a necessary condition for the non-transitivity of a generic area-preserving diffeomorphism. In fact, a consequence of Theorem \ref{th:continuo} is that the familly of aperiodic invariant continua which are minimal with respect to the property of being annular is pairwise disjoint (we call these continua \emph{frontiers}, see \cite{k-m} for details). This allows a sort of decomposition of the dynamics in terms of the aperiodic invariant continua. 
Similar concepts appear in the work of J\"ager \cite{jaeger} (where the word \emph{circloid} is used instead of frontier) when studying nonwandering homeomorphisms of the torus with bounded mean motion. 

In \cite{k-m}, these observations play a fundamental role in the proof of the following result: for any $r\geq 1$, given a $C^r$-generic pair of area-preserving diffeomorphisms of a compact surface, the \emph{iterated function system} (or, equivalently, the action of the semi-group) generated by them is transitive.

We should mention that the basic idea of the proof of Theorem \ref{th:continuo} is inspired by the analogous result from \cite{franks-lecalvez} in the case where the surface is a sphere.

This article is organized as follows. In \S1-5 we recall some background and results about ideal boundary points, continua, Lefschetz numbers and indices and we prove some elementary facts; in \S5 we prove our main theorem, and a corollary about rotation numbers is mentioned; in \S6 we state a version of the theorem for non-orientable surfaces, with an outline of the proof.

\subsection*{Acknowledgments}
I am grateful to M. Nassiri for motivating this problem, as well as L. N. Carvalho, J. Franks and E. R. Pujals for useful discussions.


\subsection{Ideal boundary, continua, and complementary domains}

If $U$ is a non-compact surface, a \emph{boundary representative} of $U$ is a sequence $P_1\supset P_2\supset\cdots$ of connected unbounded (\ie not relatively compact) open sets in $U$ such that $\bd_U P_n$ is compact for each $n$ and for any compact set $K\subset U$, there is $n_0>0$ such that $P_n\cap K=\emptyset$ if $n>n_0$. Two boundary representatives $\{P_i\}$ and $\{P_i'\}$ are said to be equivalent if for any $n>0$ there is $m>0$ such that $P_m\subset P_n'$, and vice-versa. The \emph{ideal boundary} of $U$ is defined as the set $\ib U$ of all equivalence classes of boundary representatives. We denote by $U^*$ the space $U\cup \ib U$ with the topology generated by sets of the form $V \cup V'$, where $V$ is an open set in $U$ such that $\bd_U V$ is compact, and $V'$ denotes the set of elements of $\ib U$ which have some boundary representative $\{P_i\}$ such that $P_i\subset V$ for all $i$. We call $U^*$ the \emph{ideal completion} of $U$.
Any homeomorphism $f\colon U\to U$ extends to a homeomorphism $f^*\colon U^*\to U^*$ such that $f^*|_{U} = f$. If $U$ is orientable and $\ib U$ is finite, then $U^*$ is a compact orientable boundaryless surface. See \cite{richards} and \cite{ahlfors-sario} for more details. 

From now on, $S$ will denote a compact orientable surface. Let $U$ be an open connected subset of $S$. For each $p^*\in \ib U$, we write $Z(p^*)$ for the set $\cl_S(\bigcap_{V} V\cap U)$ where the intersection is taken over all neighborhoods $V$ of $p^*$ in $U^*$. It is easy to see that $Z(p^*)$ is a compact, connected, nonempty set (see \cite{mather-cara}).

\begin{definition}\label{def:regular} We say that $U\subset S$ is a \emph{complementary domain} if it is a connected component of the complement of some compact connected subset of $S$.
\end{definition}

The next proposition is a direct consequence of \cite[Lemma 2.3]{mather-area}. 

\begin{proposition} If $U$ is a complementary domain in $S$, then it has finitely many ideal boundary points.
\end{proposition}

If $\ib U$ is finite, for each $p^*\in \ib U$ we may choose a neighborhood $V$ of $p$ such that $\ol{V}$ is homeomorphic to a closed disk, and such that $\ol{V}\cap \ib U = \{p\}$. Thus $V\sm\{p\}$ is a topological annulus in $S$. And, unless $U$ is a topological disk, the boundary of $V$ is an essential simple closed curve in $S$. From this, we have

\begin{proposition} \label{pro:regular-surface} If $U$ is a complementary domain in $S$, then $\ib U$ is finite, and there is a compact bordered surface  $S_U\subset U$ such that $U\sm S_U$ has finitely many connected components, each of which is homeomorphic to an open annulus.
\end{proposition}

\begin{corollary} \label{coro:regular-boundary} If $U$ is a complementary domain in $S$, then $\bd U$ has finitely many connected components.
\end{corollary}
\begin{proof}
Choose $K\subset U$ such that $U\sm K$ is a finite union of disjoint annuli. If $A$ is a connected component of $U\sm K$, then $\bd A\sm U$ is connected (since it is $Z(p^*)$ for some $p^*\in \ib U$), and $\bd U = \bigcup_A \bd A\sm U$, where the union is taken over all connected components $A$ of $U\sm K$. Since these are finitely many, the claim follows.

\end{proof}

We remark that the number of boundary components of $U$ may be smaller than the number of ideal boundary points, since the sets $Z(p^*)$, $p^*\in \ib U$ need not be disjoint.

\subsubsection{Continua}
By a \emph{continuum} we mean a compact connected set.

\begin{proposition} \label{lem:inf-cc-disk} Let $K$ be a continuum and $\mathcal{U}$ the family of all connected components of $S\sm K$. Then all but finitely many elements of $\mathcal{U}$ are simply connected. 
\end{proposition}
\begin{proof} We consider two cases. First suppose that for some $U\in \mc{U}$, there is a simple closed curve $\gamma$ which is homotopically nontrivial in $U$ but trivial in $S$. Let $D$ be the topological disk bounded by $\gamma$ in $S$. Since $\gamma$ is nontrivial in $U$, there is some point of $K$ in $D$. Since $K$ is connected and $\gamma \subset S\sm K$, it follows that $K\subset D$. Thus if $U'\in \mc{U}$ and $U'\neq U$, then $U'\subset D$. From this we conclude that $U'$ is simply connected. Indeed, if $\gamma'$ is a homotopically nontrivial simple closed curve in $U'$, then by a similar argument it bounds a disk $D'\subset D$ which intersects $K$, so $K\subset D'$. But this implies that $S\sm D'\subset U'$ (because it is connected) so $U'=U$, a contradiction. Therefore, all but one element of $\mc{U}$ are simply connected.

Now suppose that for every $U\in \mc{U}$, if $\gamma$ is homotopically nontrivial in $U$ then it is also homotopically nontrivial in $S$, and assume that there are infinitely many complementary domains $U_1, U_2, \dots$ of $K$ which are not simply connected. For each $U_i$, let $\gamma_i$ be a simple homotopically nontrivial simple closed curve in $U_i$. By our assumption, $\gamma_i$ is also nontrivial in $S$. The curves $\{\gamma_i:i\in \N\}$ are pairwise disjoint, so there must be infinitely many of them in the same homotopy class of $S$. But if, say, $\gamma_1$, $\gamma_2$ and $\gamma_3$ are all homotopic and disjoint, there are two disjoint annuli $A_1$ and $A_2$ such that (up to reordering the indices) $\bd A_1 = \gamma_1\cup \gamma_2$ and $\bd A_2=\gamma_2\cup \gamma_3$. Since the boundary of $A_1$ contains points of two different connected components of $S\sm K$, it is clear that $A_1$ must intersect $K$. Since $K$ is connected, it follows that $K\subset A_1$. But with the same argument we also conclude that $K\subset A_2$, a contradiction. This completes the proof.
\end{proof}

\subsubsection{Annular continua}

\begin{definition} \label{def:annular} A continuum $K\subset S$ is said to be \emph{annular} if it has a neighborhood $A\subset S$ homeomorphic to an open annulus such that $A\sm K$ has exactly two components, both homeomorphic to annuli. We call any such $A$ an \emph{annular neighborhood} of $K$. 

\end{definition}

This definition is equivalent to saying that $K$ is the intersection of a sequence $\{A_i\}$ of closed topological annuli such that $A_{i+1}$ is an essential subset of $A_i$ (\ie it separates the two boundary components of $A_i$), for each $i\in \N$. 

%
%

\subsection{Indices and Lefschetz number}

If $f$ is a homeomorphism and $D$ is a closed topological disk without fixed points in its boundary, we denote by $\mathrm{Ind}_f(D)$ the fixed point index of $f$ in $D$. (see \cite{dold}). If there are finitely many fixed points of $f$ in $D$, then $\mathrm{Ind}_f(D)$ is equal to the sum of the Lefschetz indices of these fixed points. If $D_1,\dots,D_n$ are disjoint disks such that the set of fixed points of $f$ is contained in the interior of their union, then we have the Lefschetz formula:
$$\sum_{i=1}^n \mathrm{Ind}_f(D_i) = L(f)$$
where $L(f)$ denotes the Lefschetz number of $f$.

\begin{lemma} \label{lem:lefschetz} Let $S$ be an orientable closed surface with Euler characteristic $\chi(S)\leq 0$. Then, for any homeomorphism $f\colon S\to S$ there is $n>0$ such that the Lefschetz number of $f^n$ is non-positive: $L(f^n)\leq 0$.
\end{lemma}
\begin{proof} When $\chi(S)<0$, a proof can be found in \cite{xia-area}. If $\chi(S)=0$, then $S\simeq \T^2$, and the automorphism induced by $f$ on $H_1(S,\Q)$ can be represented by a matrix $A\in \SL(2,\Z)$. It is well known that any such matrix is either periodic ($A^{n}=I$ for some $n>0$, so $\tr(A^n)=2$), parabolic (and then $\tr(A^2)= 2$) or hyperbolic (and then $\tr(A^2)>2$). In either case, there is $n$ such that $L(f^n)= 2-\tr(A^n)\leq 0$.
\end{proof}

\subsection{Wandering points}
Given a homeomorphism $f\colon S\to S$, we say that a nonempty open set $U$ is \emph{wandering} if $f^n(U)\cap U=\emptyset$ for all $n>0$ (or, equivalently, for all $n\neq 0$). We denote by $\Omega(f)$ the set of non-wandering points of $f$. That is, the (compact, invariant) set of points which have no wandering neighborhood. 

\begin{remark}\label{rem:wandering} We will use the following observations several times: 
\begin{enumerate}
\item If $\Omega(f)=S$, then $\Omega(f^n)=S$. To see this, given a nonempty open set $U_0$ we can define recursively $U_{i+1}= f^{k_{i+1}}(U_i)\cap U_i$ where $k_{i+1}>0$ is chosen such that the intersection is nonempty. Then there are integers $i_1<i_2<\cdots<i_n$ such that $k_{i_1}=k_{i_2}=\cdots = k_{i_n} (\mathrm{mod}\, n)$, so that $k_{i_1}+\cdots+k_{i_n}=mn$ for some $m>0$, and it is easy to verify that $f^{mn}(U_0)\cap U_0\neq \emptyset$.
\item If $\Omega(f)=S$ and $\{U_i\}_{i\in\N}$ is a family of pairwise disjoint open sets which are permuted by $f$ (\eg the connected components of the complement of a compact periodic set) then each $U_i$ is periodic for $f$.
\end{enumerate}
\end{remark}

\begin{lemma} \label{lem:index-bd} Let $D\subset S$ be a topological open disk and $f\colon \ol{D}\to \ol{D}$ a homeomorphism. Suppose that there is a neighborhood of $\bd D$ in $\ol{D}$ which does not contain the positive or the negative orbit of any wandering open set, and $f$ has no fixed points in $\bd D$. Then the index of the set of fixed points of $f$ in $D$ is $1$. In other words, there is a closed topological disk $D'$ which contains all fixed points of $f$ in $D$, such that $\mathrm{Ind}_f(D')=1$.
\end{lemma}
\begin{proof} Since it contains no fixed points, $\bd D$ is not reduced to a single point. By a theorem of Cartwright and Littlewood \cite{cartwright-littlewood-2} (see also \cite[Proposition 2.1]{franks-lecalvez}), the extension $\hat{f}$ of $f|_D$ to the prime ends compactification $\hat{D}$ of $D$ has no fixed points in the boundary circle $\bd{\hat{D}}$. Thus $\hat{f}$ is orientation-preserving, and $\mathrm{Ind}_{\hat{f}}(\hat{D})=1$, and since fixed points of $\hat{f}$ are in a compact subset of $D$, we can choose a closed disk $D'\subset D$ containing all fixed points of $\hat{f}$, so that $\mathrm{Ind}_{\hat{f}}(D')=\mathrm{Ind}_{\hat{f}}(\hat{D})=1$. But since $D'\subset D$ and $\hat{f}|_{D'}=f|_{D'}$, it follows that $\mathrm{Ind}_{\hat{f}}(D')=\mathrm{Ind}_{f}(D')$ and we are done.
\end{proof} 

\subsection{Main theorem}

We begin with a brief outline of the proof. The idea is to generalize the index argument used in \cite{franks-lecalvez} for the case of the sphere. However, to do that we need to modify the underlying manifold: we consider the (possibly infinitely many) connected components of $S\sm K$. The non-wandering hypothesis guarantees that these components are permuted by $f$. Since these components are complementary domains, they have finitely many ends. Next we ``remove'' every nontrivial component (except for a neighborhood of its boundary), leaving a bordered submanifold $N$ of $S$ which is a neighborhood of $K$. We can modify  $f|_N$ obtaining a map which coincides with $f$ in a neighborhood of $K$, but which leaves the boundary of $N$ invariant. After collapsing the boundary circles of $N$ to points, we obtain a new compact surface containing $K$, and a homeomorphism which has no periodic points on $K$, and by a Lefschetz index argument we conclude that this surface can only be a sphere. From this we conclude easily that $K$ is annular.

\begin{remark}
If $K$ is an aperiodic invariant continuum and $K\neq S$, then Theorem \ref{th:continuo} implies that $K$ is annular. Following \cite[\S3]{franks-lecalvez} (using a small annular neighborhood $A$ of $K$, and lifting $f$ to the universal covering of $A$) one can define the rotation set $\rho_f(K)\subset \R$ (which is defined modulo integer translations). 
Now, with almost no modifications, the proof of \cite[Proposition 5.2]{franks-lecalvez} remains valid. Thus we obtain the following

\begin{corollary} If $f\colon S\to S$ is an area preserving homeomorphism and $K\subsetneq S$ is an invariant continuum with no periodic points, then $K$ is annular,  $\rho_{f}(K)$ consists of a single irrational number $\alpha$, and the rotation numbers in the prime ends from both sides of $K$ coincide (up to a sign change) with $\alpha$. 
\end{corollary}
\end{remark}

\subsubsection{Proof of Theorem \ref{th:continuo}}
We may assume that $f$ is orientation-preserving (otherwise consider $f^2$ instead of $f$).
If $K=S$ and $f$ has no periodic points, then $S=\T^2$ by the Lefschetz theorem, and we are done. Now suppose that $f$ has no periodic points in $K$ and $K\neq S$. We need to show that $K$ is annular.

Consider the family $\mc{V}$ of connected components of $S\sm K$ which are not topological disks, which is finite by Proposition \ref{lem:inf-cc-disk}. Since open sets are nonwandering, each element of $\mc{V}$ is periodic by $f$. Choosing a power of $f$ instead of $f$ we may (and we do from now on) assume that each element of $\mc{V}$ is fixed by $f$. 

Since each  $V\in \mc{V}$ is a complementary domain, by Proposition \ref{pro:regular-surface} we can choose a compact surface with boundary $S_V\subset V$ such that $V\sm S_V$ has finitely many components, all of which are annuli.  

Given $V\in \mc{V}$, the ideal boundary points of $V$ are periodic by $(f|_V)^*$, so by taking a power of $f$ instead of $f$ we may assume that they are in fact fixed. This implies that if $\gamma$ is a sufficiently small closed loop in $V$ which bounds a disk containing $p^*$ in $V^*$, then $f(\gamma)$ is homotopic to $\gamma$ in $V$ (and thus in S). Moreover, $f(Z(p^*)) = Z(p^*)$ for any $p^*$ in $\ib V$. Note also that $Z(p^*)\subset K$ for all $p^*\in \ib V$.

Let $A_1,\dots, A_n$ be the connected components of $V\sm S_V$. Each $A_i$ is a topological annulus, whose boundary in $S$ is given by a loop $\gamma_i$ and the continuum $Z_i = \ol{A}_i\cap K$ (which is $Z(p^*)$ for some $p^*\in V^*$). Since $f(Z_i)=Z_i$, if $\sigma_i \subset A_i$ is an essential simple closed curve close enough to $Z_i$, we have that $f(\sigma_i)\subset A_i$. Since $f(\sigma_i)$ is homotopic to $\sigma_i$ in $A_i$, there exists a homeomorphism  $h_i\colon A_i\to A_i$ which maps $f(\sigma_i)$ to $\sigma_i$ and which is the identity in a neighborhood of the boundary of $A_i$; furthermore, we may assume that $h_i(x)=f^{-1}(x)$ for $x\in f(\sigma_i)$ (see \cite{epstein}). Extending $h_i$ to the identity outside $A_i$, and letting $\til{f}=h_1\dots h_nf$, we get an orientation preserving homeomorphism such that $\til{f}(x)=f(x)$ for $x\in S\sm \cup_i A_i$ and $\til{f}(\sigma_i)=\sigma_i$. If $\til{S}_V$ is the surface bounded by $\sigma_1,\dots, \sigma_n$ which intersects $S_V$, we have that $\til{f}(\til{S}_V)=\til{S}_V$ and $\til{f}$ is the identity on the boundary of $\til{S}_V$. We do this for each $V\in \mc{V}$, and finally we consider the boundaryless compact surface $\til{S}$ obtained by collapsing each boundary circle of $S\sm \inter\cup_{V\in \mc{V}}{\til{S}_V}$ to a point, and the induced homeomorphism which we still call $\til{f}$, for which these points are fixed (see figure \ref{fig1}). This new surface contains $S\sm \cup\mc{V}$, and $\til{f}$ coincides with $f$ on that set. Each $V\in \mc{V}$ was replaced by a (finite) union of one or more invariant topological disks, and the boundary of each of these disks is contained in $K$ (and hence, it contains no periodic points). 
\begin{figure}[ht!]
\centering{\resizebox{0.6\textwidth}{!}{\includegraphics{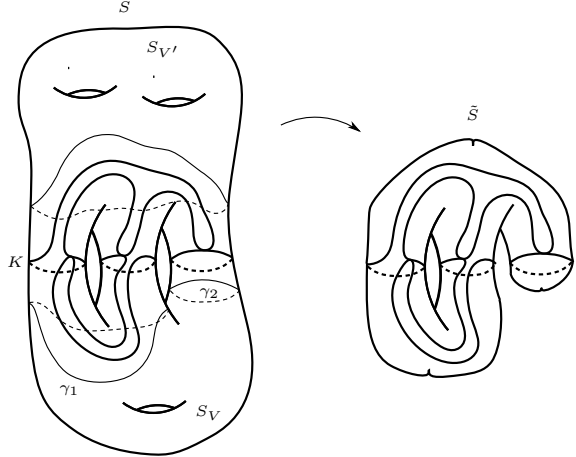}}}
  \caption{The complement of $K$ in $\til{S}$ consists of topological disks} \label{fig1}
\end{figure}

Since $\mc{V}$ consists of all components of $S\sm K$ which are not disks, from our construction we see that all components of $\til{S}\sm K$ are topological disks. 

Suppose that $\chi(\til{S})\leq 0$. Then by Lemma \ref{lem:lefschetz} there is $n$ such that $L(\til{f}^n)\leq 0$. 
Let $D$ be a connected component of $\til{S}\sm K$ such that $f^n(D)=D$. We know that $\til{f}^n$ coincides with $f^n$ in a neighborhood of $\bd D\subset K$, so the fact that $f^n$ has no wandering points (and no fixed points in $K$) implies that the hypotheses of Lemma \ref{lem:index-bd} hold. Hence, the fixed point index of $\til{f}^n$ in $D$ must be $1$ (in particular, $D$ contains a fixed point).
From this, it follows that there are finitely many $\til{f}^n$-invariant components in $\til{S}\sm K$. In fact, if there were infinitely many, then one could find a sequence of fixed points accumulating in $K$, which contradicts the aperiodicity of $K$. Moreover, we may assume that there is at least one such component (by starting with an appropriate power of $f$ instead of $f$).

Since $\til{f}^n$ has no fixed points in $K$, denoting the components of $\til{S}\sm K$ which are $\til{f}^n$-invariant by $D_1,\,\dots,\, D_k$, we have from the Lefschetz formula $$L(\til{f}^n) = \sum_{i=1}^k \mathrm{Ind}_{\til{f}^n}(D_i) = k \geq 1,$$ which contradicts our choice of $n$. From this we conclude that $\chi(\til{S})>0$, hence $\til{S}$ is a sphere. 

But then, since $\til{f}$ preserves orientation, $L(\til{f}^m)=\chi(\til{S})=2$ for all $m$. This implies that $\til{S}\sm K$ consists of exactly two components  $D_1$ and $D_2$. In fact, if there were more than two such components, it would be possible to choose $m$ such that $\li{f}^m$ leaves three or more of those components fixed, so that, repeating our previous argument, $L(\til{f}^m)\geq 3$, contradicting our previous claim.

Since $D_1$ and $D_2$ are topological disks, each of them is the union of an increasing sequence of closed topological disks, so that $K$ is the intersection of a decreasing sequence of annuli $\{A_n\}$. These annuli are eventually contained in any neighborhood of $K$, which means that, for some $n_0$,  $\{A_n\}_{n\geq n_0}$ is a decreasing sequence of annuli in the original surface $S$, and $\cap_{n\geq n_0} A_n=K$. Thus $K$ is annular in $S$. This completes the proof. \qed

\subsection{Non-orientable case of Theorem \ref{th:continuo}}

\begin{corollary} \label{coro:non-orientable} Let $f\colon S\to S$ be a homeomorphism of the closed non-orientable surface $S$, such that $\Omega(f)=S$. If $K$ is an $f$-invariant continuum, then one of the following holds:
\begin{enumerate}
\item $f$ has a periodic point in $K$;
\item $K$ is annular;
\item $K$ is the intersection of a nested sequence of M\"obius strips;
\item $K=S=$ Klein bottle.
\end{enumerate}
\end{corollary}
\begin{proof}
We consider the oriented double covering $\pi\colon \hat{S}\to S$, and a lift $\hat{f}\colon \hat{S}\to \hat{S}$ of $f$. Since $f$ has no wandering points, that must be true of $\hat{f}$ as well. In fact, if $\hat{U}\subset \hat{S}$ is a sufficiently small open set, then $\pi^{-1}(\pi(\hat{U})) = \hat{U}\cup \hat{U}'$ where the union is disjoint and $\hat{U}'$ is homeomorphic to $\hat{U}$. If $n> 0$ is such that $f^n(\pi(U))\cap \pi(U)\neq \emptyset$ then either $\hat{f}^n(\hat{U})\cap \hat{U}\neq \emptyset$ or $\hat{V}'=\hat{f}^n(\hat{U})\cap \hat{U}'\neq \emptyset$. If the latter case holds, then again $\pi^{-1}(\pi(\hat{V}'))$ is the disjoint union of $\hat{V}'$ and $\hat{V}$, where $\hat{V}\subset \hat{U}$, and there is $m>0$ such that $\hat{f}^m(\hat{V})\cap \hat{V}' \neq \emptyset$ (which implies that $\hat{f}^{m+n}(\hat{U})\cap \hat{U})\neq \emptyset$) or $\hat{f}^m(\hat{V})\cap \hat{V} \neq \emptyset$ (which implies $\hat{f}^m(\hat{U})\cap \hat{U}\neq \emptyset$), so $\hat{U}$ is nonwandering.

Now $\pi^{-1}(K)$ consists of either a unique connected $\hat{f}$-invariant set or two copies of $K$ which are invariant by $\hat{f}$ if the lift is chosen appropriately. Let $\hat{K}$ be one of those components (or the unique component if there is only one). If $K$ has no periodic points of $f$, then $\hat{f}$ cannot have a periodic point in $\hat{K}$, because periodic points of $\hat{f}$ project to periodic points of $f$. Thus we are in the setting of Theorem \ref{th:continuo}, and we conclude that either $\hat{K}$ is annular or $\hat{S}=\T^2$. In the latter case, it follows that $S$ is a Klein bottle. In the former case, we have a decreasing sequence of topological annuli $\{\hat{A}_i\}_{i\in \N}$ such that $\hat{K} = \bigcap_i \hat{A}_i$. The sets $\hat{A}_i$ project to a decreasing sequence of neighborhoods $\{A_i\}_{i\in \N}$ of $K$, each of which is either homeomorphic to an annulus (in which case it projects injectively) or to a M\"obius strip, and it is easy to see that $K = \bigcap_i A_i$. By taking a subsequence of $\{A_i\}_{i\in \N}$ if necessary, we see that either $(2)$ or $(3)$ must hold. 
\end{proof}

\bibliographystyle{amsalpha} 
\bibliography{tesis}

\end{document}